\documentclass[a4paper,11pt]{article}
\usepackage[T1]{fontenc}
\usepackage{amsfonts}
\usepackage{textcomp}
\usepackage{amsmath, amsthm, amssymb, mathrsfs}
\usepackage{graphicx, url}
\usepackage{xcolor}
\usepackage[utf8]{inputenc}
\usepackage[margin=1in]{geometry}
\usepackage{hyperref}

\title{On the Simulation of General Multivariate Gamma Distributions using Dickman Approximations}
\author{Michael Grabchak\footnote{Email address: mgrabcha@charlotte.edu}\ \ and Xingnan Zhang\footnote{Email address: xzhang627@outlook.com}\\\
	{\it University of North Carolina Charlotte}}

\begin{document}

\newtheorem{prop}{Proposition}
\newtheorem{thrm}{Theorem}
\newtheorem{defn}{Definition}
\newtheorem{cor}{Corollary}
\newtheorem{lemma}{Lemma}
\newtheorem{remark}{Remark}
\newtheorem{exam}{Example}
\newtheorem{algo}{Algorithm}

\newcommand{\rd}{\mathrm d}
\newcommand{\GD}{\mathrm{GD}}
\newcommand{\MD}{\mathrm{MD}}
\newcommand{\MVP}{\mathrm{MVP}}
\newcommand{\rD}{\mathrm D}
\newcommand{\rE}{\mathrm E}
\newcommand{\rF}{\mathrm F}
\newcommand{\rP}{\mathrm P}
\newcommand{\rG}{\mathrm G}
\newcommand{\rM}{\mathrm M}
\newcommand{\ts}{\mathrm{TS}^p_{\alpha,d}}
\newcommand{\Exp}{\mathrm{Exp}}
\newcommand{\ID}{\mathrm{ID}}
\newcommand{\tr}{\mathrm{tr}}
\newcommand{\iid}{\stackrel{\mathrm{iid}}{\sim}}
\newcommand{\eqd}{\stackrel{d}{=}}
\newcommand{\approxd}{\stackrel{d}{\approx}}
\newcommand{\cond}{\stackrel{d}{\rightarrow}}
\newcommand{\conv}{\stackrel{v}{\rightarrow}}
\newcommand{\conw}{\stackrel{w}{\rightarrow}}
\newcommand{\conp}{\stackrel{p}{\rightarrow}}
\newcommand{\confdd}{\stackrel{fdd}{\rightarrow}}
\newcommand{\simp}{\stackrel{p}{\sim}}
\maketitle

\setcounter{page}{1}
\pagenumbering{arabic}

\begin{abstract}
We derive a Dickman approximation for the small jumps of a large class of multivariate L\'evy processes. We then apply this approximation to develop a simulation method for the class of general multivariate gamma distributions (GMGD). A small-scale simulation study suggests that this method works very well.\\

\noindent \textbf{Keywords:} Multivariate gamma distributions; Dickman distribution; Small jumps of L\'evy processes; Simulation\\
\noindent \textbf{MSC:} 60G51; 60F05

\end{abstract}

\section{Introduction}

Gamma distributions and their extensions have been widely used in mathematical finance. Perhaps most prominent is the use of variance gamma and bilateral gamma distributions to model financial returns, see, e.g.,\ \cite{Madan:Seneta:1990}, \cite{Cont:Tankov:2004}, \cite{Kuchler:Tappe:2008}, \cite{Sabino:2020}, and the references therein.  In practice, one typically works with a basket of assets and needs a multivariate distribution to jointly model the returns. We focus on the class of general multivariate gamma distributions (GMGD), which is a large class that contains several other models, including the multivariate gamma distributions of \cite{Perez-Abreu:Stelzer:2014} as well as important subclasses of the multivariate Thorin class of generalized gamma convolutions (GGC) \cite{Barndorff-Nielsen:Maejima:Sato:2006} \cite{Laverny:etal:2021} and of the so-called Class M \cite{Maejima:2015}. GMGD can be used to model baskets of returns directly or it can be first combined with Brownian motion through the process of multivariate subordination, see \cite{Barndorff-Nielsen:Pedersen:Sato:2001} or \cite{Buchmann:etal:2017}. Either way, simulation is an important tool to implement Monte Carlo methods for option pricing and risk estimation. In this paper we develop an approximate simulation method for GMGD, which is based on a multivariate Dickman approximation. The idea is to consider a GMGD L\'evy process and to model its small jumps using a multivariate Dickman L\'evy process and its large jumps by a compound Poisson process. We will apply this to option pricing in a future work.

For the purposes of simulation, small jumps of L\'evy processes are often approximated by simpler processes, see, e.g., Chapter 6 in \cite{Cont:Tankov:2004}. The most common process used is Brownian motion, where the approximation is justified by a limit theorem proved in \cite{Asmussen:Rosinski:2001} for the univariate case and extended to the multivariate case in \cite{Cohen:Rosinski:2007}. However, this approximation does not hold for gamma and related distributions. In \cite{Covo:2009} it was shown that, in the univariate case, the small jumps of such L\'evy processes can be approximated by Dickman L\'evy processes. What has been missing from the literature is a Dickman approximation in the multivariate case. In fact, the multivariate Dickman distribution was only recently introduced in \cite{Bhattacharjee:Molchanov:2020} and it was not studied in detail until \cite{Grabchak:Zhang:2024}. In this paper, we prove a limit theorem characterizing when the small jumps of a L\'evy process can be approximated by a multivariate Dickman L\'evy process, and we show that this always holds for GMGD. Furthermore, we develop a methodology for the simulation of the large jumps in this case, which creates a complete methodology for approximate simulation of GMGD.

The rest of this paper is organized as follows. In Section \ref{sec: background} we review basic facts about L\'evy processes and Dickman distributions, and we introduce the slightly more general class of multivariate $\epsilon$-Dickman distributions. In Section \ref{sec: main} we give our main results for approximating the small jumps of certain multivariate L\'evy processes by multivariate Dickman L\'evy processes and discuss how to use this for simulation. In Section \ref{sec: gen gamma} we formally introduce GMGD and give detailed results about simulation using the Dickman approximation. We also perform a small-scale simulation study, which suggests that this method works very well.  Proofs are postponed to Section \ref{sec: proofs}.

Before proceeding, we introduce some notation. We write $\mathbb R^d$ to denote the set of $d$-dimensional column vectors equipped with the usual inner product $\langle\cdot,\cdot\rangle$ and the usual norm $|\cdot|$. We write $\mathbb S^{d-1}=\{x\in\mathbb R^d:|x|=1\}$ to denote the unit sphere in $\mathbb R^d$, and we write $\mathfrak{B}(\mathbb{R}^{d})$ and $\mathfrak{B}(\mathbb{S}^{d-1})$ to denote the Borel sets on $\mathbb R^d$ and $\mathbb S^{d-1}$, respectively. For a distribution $\mu$ on $\mathbb R^d$, we write $X\sim \mu$ to denote that $X$ is a random variable with distribution $\mu$ and $X_1,X_2,\dots\iid \mu$ to denote that $X_1,X_2,\dots$ are independent and identically distributed (iid) random variables with distribution $\mu$. We write $U(a,b)$ to denote the uniform distribution on $(a,b)$, $\Exp(\lambda)$ to denote the exponential distribution with rate $\lambda$, and $\mathrm{Pois}(\lambda)$ to denote the Poisson distribution with mean $\lambda$. We write $1_A$ to denote the indicator function of set $A$ and $\delta_a$ to denote the point-mass at $a$. For any $a\in\mathbb R$ and $B\subset\mathbb R^d$ we write $a B = \{a y: y \in B\}$, and for any $C \subset \mathbb{S}^{d-1}$ and $0\le a<b<\infty$ we write
\begin{equation*}
	(a,b]C = \left\{x \in \mathbb{R}^d:|x| \in (a,b], \frac{x}{|x|} \in C \right\}.
\end{equation*}

\section{Background}\label{sec: background}

In this sections we review basic facts about L\'evy processes and Dickman distributions, and we introduce the slightly more general class of multivariate $\epsilon$-Dickman distributions.

\subsection{L\'evy Processes}

The characteristic function of an infinitely divisible distribution $\mu$ on $\mathbb R^d$ can be written in the form $\hat\mu(z) = \exp\{C_\mu(z)\}$, where
\begin{eqnarray*}\label{eq: char func inf div}
	C_{\mu}(z) =  -\langle z,Az\rangle + i \left\langle b, z \right\rangle +
	\int_{\mathbb{R}^d}\left(e^{i\left\langle z, x\right\rangle } - 1 - i\left\langle z, x\right\rangle 1_{[|x|\le1]}\right) M(\rd x), 
	\ \ \ z \in \mathbb{R}^d,
\end{eqnarray*}
$A$ is a $d\times d$-dimensional covariance matrix called the Gaussian part, $b\in\mathbb R^d$ is the shift, and $M$ is the L\'evy measure, which is a Borel measure on $\mathbb R^d$ satisfying
\begin{equation*}\label{eq: levy measure equation gen}
M(\{0\}) = 0 \text{  and  } \int_{\mathbb{R}^d}(|x|^2 \wedge 1) M(\rd x) < \infty.
\end{equation*}
The parameters $A$, $M$, and $b$ uniquely determine this distribution and we write $\mu=\ID(A,M,b)$. We call $C_\mu$ the cumulant generating function (cgf) of $\mu$. Associated with every infinitely divisible distribution $\mu$ is a L\'evy process $\{X_t:t\ge0\}$, where $X_1\sim\mu$. In this context, the L\'evy measure governs the jumps of the process. Specifically, for any $B\in\mathfrak B(\mathbb R^d)$, $M(B)$ is the expected number of jumps that the process has in the time interval $[0,1]$ that fall inside set $B$. 

A L\'evy process has finite variation if and only if $A=0$ and $M$ satisfies the additional condition 
\begin{equation*}\label{eq: levy measure equation finite var}
\int_{\mathbb{R}^d}(|x| \wedge 1) M(\rd x) < \infty.
\end{equation*}
In this case, the cgf can be written in the form
\begin{eqnarray}\label{eq: char func inf div finite variation}
	C_\mu(z)=  i \left\langle \gamma, z \right\rangle +
	\int_{\mathbb{R}^d}\left(e^{i\left\langle z, x\right\rangle } - 1\right) M(\rd x),
	 \ \ \ z \in \mathbb{R}^d,
\end{eqnarray}
where $\gamma = b-\int_{|x|\le1}x M(\rd x) \in \mathbb{R}^d$ is the drift, and we write $\mu = \ID_0(M, \gamma)$. For more on infinitely divisible distributions and L\'evy processes see \cite{Sato:1999} or \cite{Cont:Tankov:2004}.

When discussing weak convergence of infinitely divisible distributions and L\'evy processes, the concept of vague convergence is fundamental. A portmanteau theorem giving several statements that are equivalent to vague convergence can be found in, e.g.,\ \cite{Barczy:Pap:2006}. The definition is as follows.

\begin{defn}\label{defn: vague conv}
Let $M_0,M_1,M_2,\dots$ be a sequence of L\'evy measures on $\mathbb R^d$. We say that $M_n$ converges vaguely to $M_0$ and write $M_n\conv M_0$ as $n\to\infty$ if $\lim\limits_{n\to\infty} \int_{\mathbb{R}^d} f(x) M_n(\rd x)=\int_{\mathbb{R}^d} f(x) M_0(\rd x)$, for every $f:\mathbb R^d\mapsto\mathbb R$ that is bounded, continuous, and vanishing on a neighborhood of $0$.
\end{defn}

\subsection{Dickman Distributions}

For $\epsilon>0$, a random variable $X$ on $\mathbb R$ is said to have a generalized $\epsilon$-Dickman distribution if
$$
X\eqd U^{1/\theta}(X+\epsilon),
$$
where $\theta\ge0$ and $U\sim U(0,1)$ is independent of $X$ on the right side. We denote this distribution by $\GD^\epsilon(\theta)$. For $\theta=0$, we interpret $\GD^\epsilon(0)$ as the distribution concentrated at $0$. This is motivated by the fact that, since $U\in(0,1)$ with probability $1$, we have $U^{1/\theta}\to0$ with probability $1$ as $\theta\downarrow0$. When $\epsilon=1$, the distribution $\GD^1(\theta)$ is called a generalized Dickman distribution and when both $\epsilon=1$ and $\theta=1$, it is just called the Dickman distribution. Many properties and applications for the case $\epsilon=1$ are discussed in the surveys \cite{Penrose:Wade:2004}, \cite{Molchanov:Panov:2020},  \cite{Grabchak:Molchanov:Panov:2022}, and the references therein. The case $\epsilon\ne1$ is discussed in \cite{Grabchak:2025} and \cite{Grahovac:etal:2024}. A multivariate generalization of the Dickman distribution was recently introduced in \cite{Bhattacharjee:Molchanov:2020} and it was further studied in \cite{Grabchak:Zhang:2024}, where many properties were derived and several approaches for simulation were developed. We now introduce a slight generalization of this model.

For $\epsilon>0$, a random variable $X$ on $\mathbb R^d$ is said to have a multivariate $\epsilon$-Dickman distribution if
\begin{eqnarray}\label{eq: relation for MD eps}
X\eqd U^{1/\theta}(X+ \epsilon \xi),
\end{eqnarray}
where $\theta\ge0$ and $X, \xi, U$ are independent on the right side with $U\sim U(0,1)$ and $\xi\sim \sigma_1$ for some probability distribution $\sigma_1$ on $\mathbb S^{d-1}$. Again, for $\theta=0$, we interpret $X$ as having a distribution concentrated on $0\in\mathbb R^d$, and note that, in this case, the distribution $\sigma_1$ does not matter. Let $\sigma=\theta\sigma_1$ and note that $\theta=\sigma(\mathbb S^{d-1})$ and that for $\theta\ne0$, $\sigma_1 = \sigma/\sigma(\mathbb S^{d-1})$. Thus, there is no loss of information when working with $\sigma$ instead of $\theta$ and $\sigma_1$. We write $\MD^\epsilon(\sigma)$ to denote this distribution and we refer to $\sigma$ as the spectral measure. It is readily checked that any nonzero finite Borel measure on $\mathbb S^{d-1}$ can serve as the spectral measure of a multivariate $\epsilon$-Dickman distribution. When $\epsilon=1$, $\MD^1(\sigma)$ reduces to the multivariate Dickman distribution studied in \cite{Bhattacharjee:Molchanov:2020} and \cite{Grabchak:Zhang:2024} and, for simplicity of notation, we write $\MD(\sigma)$ in this case.  The generalized $\epsilon$-Dickman distribution $\GD^\epsilon(\theta)$ corresponds to $\MD^\epsilon(\sigma)$ with dimension $d=1$ and $\sigma=\theta\delta_1$. From \eqref{eq: relation for MD eps} it is easily checked that  $\epsilon$-multivariate Dickman distributions belong to the class of multivariate Vervaat perpetuities, which were introduced in \cite{Grabchak:Zhang:2024}.

\begin{lemma}\label{lemma: Levy measure of MD}
1. For any $\epsilon,\gamma>0$, if $X\sim \MD^\epsilon(\sigma)$, then
\begin{eqnarray}\label{eq: scaling for MD}
\frac{1}{\gamma}X \sim \MD^{\epsilon/\gamma}(\sigma).
\end{eqnarray}
2. If $\mu=\MD^\epsilon(\sigma)$, where $\epsilon>0$ and $\sigma$ is a finite Borel measure on $\mathbb S^{d-1}$, then  $\mu= \ID_0(D^\epsilon, 0)$, where
\begin{eqnarray}\label{eq: Levy meas epsilon}
D^\epsilon(B) = \int_{\mathbb S^{d-1}} \int_0^\epsilon 1_B(rs) r^{-1}\rd r \sigma(\rd s), \ \ \ B\in\mathfrak B(\mathbb R^d).
\end{eqnarray}
\end{lemma}

For $\epsilon=1$, the result in the second part is given in Theorem 5 of \cite{Grabchak:Zhang:2024}. Recall that $D^\epsilon(B)$ is the expected number of jumps that the L\'evy process associated with $\MD^\epsilon(\sigma)$ will have in the time period $[0,1]$, which fall inside of set $B$. Thus, \eqref{eq: Levy meas epsilon} implies that $\epsilon$ is the largest possible magnitude for a jump of this L\'evy process.

\section{Multivariate Dickman Approximations of Small Jumps}\label{sec: main}

In this section we derive a limit theorem showing that a Dickman L\'evy process can be used to approximate the small jumps of a large class of L\'evy processes. Here and throughout, when applied to L\'evy processes $\cond$ refers to weak convergence on the space $D([0,\infty),\mathbb R^d)$, which is the space of c\`adl\`ag functions from $[0,\infty)$ into $\mathbb R^d$ equipped with the Skorokhod topology. For any $\gamma\in\mathbb R^d$, we write $\gamma^*=\{t\gamma:t\ge0\}$ to denote the element of $D([0,\infty),\mathbb R^d)$ that maps $t$ to $t\gamma$. In particular, $0^*$ denotes the function that is identically zero.

Let $X=\{X_t:t\ge0\}$ be a L\'evy process with $X_1\sim\ID_0(\nu,0)$. Note that here, for simplicity, we set the drift to zero. Fix $\epsilon>0$ and consider the truncated L\'evy process $X^\epsilon=\{X_t^{\epsilon}:t\ge0\}$ obtained by removing the jumps of the process $X$, whose magnitudes exceed $\epsilon$. In this case, $X^\epsilon_1\sim\ID_0(\nu^\epsilon,0)$, where $\nu^{\epsilon}(B)= \int_{|x| \leq \epsilon} 1_B(x) \nu(\rd x)$, $B \in \mathfrak{B}(\mathbb{R}^d)$. Next, consider the scaled truncated process $\epsilon^{-1}X^{\epsilon}= \{\epsilon^{-1}X_t^{\epsilon}:t\ge0\}$ and note that all of its jumps are bounded by $1$. It is easily checked that $\epsilon^{-1}X_{1}^{\epsilon}\sim \ID_0(M^\epsilon,0)$, where
$$
M^\epsilon(B) = \int_{|x|\le\epsilon} 1_B\left(\frac{x}{\epsilon}\right) \nu (\rd x)=\int_{\mathbb{R}^d} 1_B\left(\frac{x}{\epsilon}\right) \nu^\epsilon (\rd x) = \nu^\epsilon (\epsilon B), \ \ \ B \in \mathfrak{B}(\mathbb{R}^d).
$$
Now, consider the multivariate Dickman L\'evy process $Y^1=\{Y_t^1:t \geq 0\}$ with $Y_1^1 \sim \MD(\sigma)$. We will give conditions for the scaled truncated process $\epsilon^{-1}X^{\epsilon}$ to converge to $Y^1$ in distribution. For the convergence to hold, we need $M^\epsilon\conv D^1$ as $\epsilon \downarrow 0$. We now give several statement that are equivalent to this. In the univariate case, a version of this result is given in Proposition 2.1 of \cite{Covo:2009}. As usual, for a set $C\in \mathfrak{B}(\mathbb{S}^{d-1})$, we write $\partial C$ to denote its boundary.

\begin{prop}\label{prop:equivalent condition}
The following statements are equivalent:	
	\begin{itemize}
		\item[\textbf{1.}] $M^\epsilon\conv D^1$ as $\epsilon \downarrow 0$.
		\item[\textbf{2.}] For all $0<h<1$ and all $C \in \mathfrak{B}(\mathbb{S}^{d-1})$ with $\sigma(\partial C)=0$, $\nu((\epsilon h, \epsilon]C) \rightarrow \sigma(C)\log \frac{1}{h}$ as $\epsilon \downarrow 0$.
	\item[\textbf{3.}] For all $p>0$ and all $C \in \mathfrak{B}(\mathbb{S}^{d-1})$ with $\sigma(\partial C)=0$, $\frac{1}{\epsilon^p} \int_{(0,\epsilon]C} |x|^p \nu(\rd x) \rightarrow \frac{\sigma(C)}{p}$ as $\epsilon \downarrow 0$.
		\item[\textbf{4.}] For some $p>0$ and all $C \in \mathfrak{B}(\mathbb{S}^{d-1})$ with $\sigma(\partial C)=0$, $\frac{1}{\epsilon^p} \int_{(0,\epsilon]C}|x|^p \nu(\rd x) \rightarrow \frac{\sigma(C)}{p}$ as $\epsilon \downarrow 0$.
	\end{itemize}
\end{prop} 

Note that we allow $\sigma=0$ in the above. Part of the result is the fact that, so long as any of the equivalent conditions in Proposition \ref{prop:equivalent condition} hold, for every $p>0$ we have
$$
\int_{(0,\epsilon]C} |x|^p \nu(\rd x) <\infty.
$$
We now give our main result for approximating the small jumps of a L\'evy process by a Dickman L\'evy process.

\begin{thrm}\label{thrm: convergence}
We have $\epsilon^{-1}{X^{\epsilon}} \cond  Y^1$ as $\epsilon\downarrow0$ if and only if any of the equivalent conditions in Proposition \ref{prop:equivalent condition} hold. Furthermore, when $\sigma=0$, this is equivalent to  $\epsilon^{-1}X^{\epsilon} \cond 0^*$ as $\epsilon\downarrow0$, and when
$\sigma\ne0$, it is equivalent to 
$$
\left(\frac{\sigma(\mathbb S^{d-1})}{p\int_{|x|\le\epsilon}|x|^p \nu(\rd x)} \right)^{1/p} X^{\epsilon} \cond Y^1\mbox{ as } \epsilon\downarrow0
$$ 
for any $p>0$. 
\end{thrm}

This extends the univariate result in \cite{Covo:2009} to the multivariate case. We note that, in that paper, the result was formulated in terms of a condition that is analogous to Condition 4 with $p=1$ in our Proposition \ref{prop:equivalent condition}.

\begin{remark}
Clearly, the assumption of Theorem \ref{thrm: convergence} holds when $X\eqd Y^1$ and hence $\epsilon^{-1}{X^{\epsilon}} \cond  X^1$. Moreover, in this case, the process $X^{\epsilon}$ is an $\epsilon$-multivariate Dickman L\'evy process and applying \eqref{eq: scaling for MD} at the level of L\'evy processes gives the stronger fact that $\epsilon^{-1}X^\epsilon\eqd X^1$ for each $\epsilon>0$. In this sense, multivariate Dickman L\'evy processes are stable under a rescaling of their jumps. The fact that these processes arises in limit theorems for small jumps is analogous to how distributions that are stable under convolution are the ones that serve as the limits of sums of iid random variables (and, equivalently, the long and short time limits of L\'evy processes), see \cite{Meerschaert:Scheffler:2001}, \cite{Kallenberg:2002}, \cite{Grabchak:2015}, or \cite{Grabchak:2017}. For more on the idea of Dickman distributions being stable, see \cite{Grabchak:2025}.
\end{remark}

We now give conditions that are easily checked in an important special case. The formulation is influenced by the discussion in \cite{Rosinski:Sinclair:2010}.  For a finite Borel measure $\sigma$ on $\mathbb S^{d-1}$, let $L^1(\mathbb S^{d-1},\sigma)$ be the space of Borel functions $g:\mathbb S^{d-1}\mapsto\mathbb R$ with $\int_{\mathbb S^{d-1}} |g(s)|\sigma(\rd s)<\infty$. For any nonnegative $g\in L^1(\mathbb S^{d-1},\sigma)$ define the finite measure $\sigma_g$ by $\sigma_g(B) = \int_{B} g(s)\sigma(\rd s)$ for $B\in\mathfrak B(\mathbb S^{d-1})$. 

\begin{cor}\label{cor:density}
Assume that $\nu$ is of the form
$$
\nu(B) = \int_{\mathbb{S}^{d-1}} \int_0^{\infty} 1_{B}(rs) \rho(r,s) \rd r \sigma(\rd s), \ \ \ B \in \mathfrak{B}(\mathbb{S}^{d-1}),
$$
where $\sigma$ is a finite Borel measure on $\mathbb{S}^{d-1}$ and $\rho:[0,\infty)\times\mathbb S^{d-1}\mapsto[0,\infty)$ is a Borel function. If there is a nonnegative $g\in L^1(\mathbb S^{d-1},\sigma)$ with $r\rho(r,\cdot) \to g(\cdot)$ in $L^1(\mathbb S^{d-1},\sigma)$ as $r \downarrow 0$, i.e.\ 
$\int_{\mathbb{S}^{d-1}} |r\rho(r, s) - g(s) | \sigma(\rd s) \to 0$ as $r \downarrow 0$, then the equivalent conditions in Proposition \ref{prop:equivalent condition} hold with $\sigma_g$ in place of $\sigma$ and $\frac{X^{\epsilon}}{\epsilon} \cond Y^1$ as $\epsilon \downarrow 0$, where $Y^1_1\sim\MD(\sigma_g)$.
\end{cor}	

We now give two ways of checking that convergence in $L^1(\mathbb S^{d-1},\sigma)$ holds. 

\begin{remark} \label{remark:uni2L1}
Assume that, for each $s\in\mathbb S^{d-1}$, $r\rho(r,s) \to g(s)$ as $r \downarrow 0$. \\
1. If $g$ is bounded and $r\rho(r,s)$ is bounded for small enough $r$, then we have convergence in $L^1(\mathbb S^{d-1},\sigma)$. This follows by dominated convergence and the fact that $\sigma$ is a finite measure.\\
2. If $r\rho(r,s) \to g(s)$ uniformly in $s$, then we have convergence in $L^1(\mathbb S^{d-1},\sigma)$. This is trivial when $\sigma=0$. To see that it holds when $\sigma\ne0$ note that, in this case, for every $\epsilon >0$ there exists a $\delta >0$ such that for every $s \in \mathbb{S}^{d-1}$ and every $r\in(0, \delta)$ we have $|r\rho(r,s)-g(s)|<\frac{\epsilon}{\sigma(\mathbb{S}^{d-1})}$. Thus, for such $r$, $\int_{\mathbb{S}^{d-1}}|r\rho(r,s)-g(s)|\sigma(\rd s)  \leq\int_{\mathbb{S}^{d-1}}\frac{\epsilon}{\sigma(\mathbb{S}^{d-1})}\sigma(\rd s)=\epsilon$.
\end{remark}

Theorem \ref{thrm: convergence} tells us that, under appropriate conditions, we can approximate the small jumps of a L\'evy process using a Dickman L\'evy process. We can then model its large jumps using a compound Poisson process. More formally, let $X=\{X_t:t\ge0\}$ be a L\'evy process with $X_1\sim\ID_0(\nu,\gamma)$. Note that we now allow for a nonzero drift. Set  $\epsilon>0$ and let $\nu^{\epsilon}(B) = \int_{|x| \leq \epsilon} 1_{B}(x) \nu(\rd x)$ and $\tilde{\nu}^{\epsilon}(B) = \int_{|x| > \epsilon} 1_{B}(x) \nu(\rd x)$ for every $B \in \mathfrak{B}(\mathbb{R}^d)$. It follows that
 \begin{eqnarray}\label{eq:decomp}
	X \eqd X^\epsilon+\tilde X^{\epsilon}+\gamma^*,
\end{eqnarray}
where $X^\epsilon=\{X^\epsilon_t:t\ge0\}$ is a L\'evy process with $X^\epsilon_1\sim\ID_0(\nu^\epsilon,0)$, $\tilde X^{\epsilon}=\{\tilde X^\epsilon_t:t\ge0\}$ is a L\'evy process with $\tilde X_1\sim\ID_0(\tilde \nu^\epsilon,0)$, and $X^\epsilon$ and $\tilde X^\epsilon$ are independent. Here, we separated $X$ into: $X^\epsilon$, the process of small jumps, $\tilde X^\epsilon$, the process of large jumps, and $\gamma^*$, the deterministic drift process. Under the assumptions of Theorem \ref{thrm: convergence}, when $\epsilon>0$ is small, we can approximate $X^\epsilon$ by $\epsilon$ times a Dickman L\'evy process $Y^1$, which gives
 \begin{eqnarray}\label{eq: approx decomp}
	X \approxd \epsilon Y^1+\tilde X^{\epsilon}+\gamma^*,
\end{eqnarray}
where $Y^1$ is independent of $\tilde X^\epsilon$. Following ideas in \cite{Cohen:Rosinski:2007} about approximations of small jumps by Brownian motion, we now verify that the error in this approximation approaches $0$ as $\epsilon\downarrow0$.

\begin{cor}\label{cor: decomp}
Let $\{X_t:t \geq 0\}$ be a L\'evy process on $\mathbb{R}^d$ with $X_1\sim\ID_0(\nu,\gamma)$ and let $Y^1$, $\tilde X^\epsilon$, and $\gamma^*$ be as above. If any of the equivalent conditions in Proposition \ref{prop:equivalent condition} hold, then
for every $\epsilon>0$, there exists a c\`adl\`ag process $Z^{\epsilon}=\{Z_t^{\epsilon}: t \geq 0\}$ such that
	\begin{equation*}
		X\eqd \epsilon Y^1 + \tilde X^{\epsilon}+\gamma^*+Z^{\epsilon},
	\end{equation*}
and, for each $T>0$,
\begin{equation}\label{eq: Zt epsilon to 0}
	\sup \limits_{t \in [0,T]} |\epsilon^{-1}Z_t^{\epsilon}| \conp 0 \ \text{as} \ \epsilon \downarrow 0.
\end{equation}
\end{cor}

This suggest that, for small $\epsilon$, we can approximate $Z^{\epsilon}$ by $0$ and we can approximately simulate from $X$ by independently simulating $Y^1$ and $\tilde X^{\epsilon}$ and then applying \eqref{eq: approx decomp}. Three methods for simulating $Y^1$ are discussed in \cite{Grabchak:Zhang:2024}. Two of the methods are approximate, while the third is exact (under mild assumptions) when simulating on a finite mesh. In this paper, we use the shot-noise method, which, while only approximate, allows us to simulate at every point on an interval. Specifically, we can simulate up to some finite time horizon $T>0$ as follows. Let $E_1,E_2,\dots\iid\Exp(\theta)$, let $U_1,U_2,\dots\iid U(0,1)$, and let $\xi_1,\xi_2,\dots\iid \sigma/\theta$, where $\theta=\sigma(\mathbb S^{d-1})$, be independent sequences of random variables. If $\Gamma_i=\sum_{k=1}^i E_i$ for $i=1,2,\dots$, then
\begin{eqnarray}\label{eq: sim MD process}
\left\{Y^1_t:0\le t\le T\right\} \eqd \left\{ \sum_{i=1}^{\infty} e^{-\Gamma_i/T} \xi_i 1_{[0,t/T]}(U_i): 0\le t\le T\right\}.
\end{eqnarray}
In practice, of course, the infinite sum must be truncated at some finite (possibly random) value, which is why this is an approximate method. 

We now turn to the simulation of $\tilde X^{\epsilon}$. Toward this end, note that $\tilde \nu^\epsilon$ is a finite measure, set $\lambda^\epsilon =  \tilde\nu^\epsilon(\mathbb R^d)$, let $\tilde \nu^\epsilon_p= \tilde \nu^\epsilon/\lambda^\epsilon$, and note that $\tilde \nu^\epsilon_p$ is a probability measure. Let $W_1,W_2,\dots\iid \tilde \nu^\epsilon_p$ and, independent of this, let $\{N(t):t\ge0\}$ be a Poisson process with rate $\lambda^\epsilon$. It is readily checked that $\tilde X^\epsilon$ is a compound Poisson process and that
 \begin{eqnarray}\label{eq: comp pois rep}
\tilde X^\epsilon =\left\{ \tilde X_t^\epsilon :t\ge0\right\} 
\eqd\left\{ \sum_{i=1}^{N(t)} W_i :t\ge0\right\}.
\end{eqnarray}
When simulating until a finite time horizon $T>0$, it is often more convenient to use the following representation. Let  $W_1,W_2,\dots\iid \tilde \nu^\epsilon_p$ and $U_1,U_2,\dots\iid U(0,1)$ be independent sequences, and, independent of these, let $N\sim\mathrm{Pois}(T\lambda_\epsilon)$. We have
 \begin{eqnarray}\label{eq: comp pois rep 2}
\left\{ \tilde X_t^\epsilon :0\le t\le T\right\} \eqd \left\{ \sum_{i=1}^{N} W_i1_{[0,t/T]}(U_i) : 0\le t\le T\right\},
\end{eqnarray}
see, e.g.,\ Section 6.1 in \cite{Cont:Tankov:2004} for details. Thus, the problem essentially reduces to finding efficient methods to simulate from $\tilde \nu^\epsilon_p$.

\section{Simulation from General Multivariate Gamma Distributions}\label{sec: gen gamma}

There are multiple ways to extend gamma distributions to the multivariate case. However, in some approaches the resulting distributions are not infinitely divisible, see the references in \cite{Perez-Abreu:Stelzer:2014}. Since infinite divisibility is important for many financial applications \cite{Cont:Tankov:2004}, we focus on an extension that maintains this property. Specifically, \cite{Perez-Abreu:Stelzer:2014} defines multivariate gamma distributions (MGD) as follows. A distribution $\mu$ on $\mathbb R^d$ is said to be MGD if $\mu=\ID_0(\nu,0)$ with
\begin{equation}\label{eq: MGD}
\nu(B) = \int_{\mathbb S^{d-1}}\int_0^\infty 1_B(r s)r^{-1}e^{-b(s)r}\rd r\sigma(\rd s), \ \ \  B \in \mathfrak{B}(\mathbb{R}),
\end{equation} 
where $\sigma$ is a finite measure on $\mathbb S^{d-1}$ and $b : \mathbb S^{d-1}\mapsto (0,\infty)$ is a Borel function. One limitation of this class is that it is not closed under taking convolutions. We consider a generalization that is closed under taking convolutions and has additional flexibility.

\begin{defn}\label{defn: GMGD}
A distribution $\mu$ on $\mathbb R^d$ is said to be a general multivariate gamma distribution (GMGD) if $\mu=\ID_0(\nu,\gamma)$, where
\begin{equation}\label{eqn:PT0S}
	\nu(B) = \int_{\mathbb{S}^{d-1}} \int_0^{\infty} 1_{B}(rs)q(r^p,s) r^{-1} \rd r \sigma(\rd s), \ \ \  B \in \mathfrak{B}(\mathbb{R}^d),
\end{equation} 
$p>0$, $\sigma$ is a finite Borel measure on $\mathbb S^{d-1}$, and $q:(0,\infty)\times\mathbb S^{d-1}\mapsto[0,\infty)$ is a Borel function such that for each $s\in\mathbb S^{d-1}$ we have $\lim_{r\downarrow0}q(r,s)=1$ and $q(\cdot,s)$ is completely monotone. 
\end{defn}

In \cite{Grabchak:2012} such distributions are called proper $p$-tempered $0$-stable distributions, see also \cite{Grabchak:2016}  for many properties. The complete monotonicity of $q$ along with our other assumptions implies that $q(r^p,s) = \int_{(0,\infty)} e^{-r^p v} Q_{s}(\rd v)$ for some measurable family $\{Q_{s}\}_{s \in \mathbb{S}^{d-1}}$ of probability measures on $(0,\infty)$ satisfying
$$
\int_{\mathbb S^{d-1}}\int_0^1 \log(1/v)Q_s(\rd v) \sigma(\rd s)<\infty,
$$
see Remark 1 and Corollary 1 in \cite{Grabchak:2012}. If we take $p=1$ and $Q_s = \delta_{b(s)}$, the L\'evy measure in \eqref{eqn:PT0S} reduces to the one in \eqref{eq: MGD}. When  $\gamma\in[0,\infty)^d$ and
$$ 
\sigma\left(\mathbb S^{d-1}\cap \left([0,\infty)^d\right)^c\right)=0
$$
the distribution $\mu$ is concentrated on $[0,\infty)^d$ and its associated L\'evy process is a multivariate subordinator, which can then be used for multivariate subordination of Brownian motion. This idea was introduced in \cite{Barndorff-Nielsen:Pedersen:Sato:2001}, see also \cite{Buchmann:etal:2017} and the references therein for a discussion of application to finance.

\begin{remark}
We note that GMGD forms an important subclass of several well-known classes of distributions. Specifically, if we allow for a Gaussian part and remove the requirement that $\lim_{r\downarrow 0}q(r,s)=1$ holds for each $s\in\mathbb S^{d-1}$, then the resulting distributions correspond to those denoted $J_{0,p}$ in \cite{Maejima:Nakahara:2009} and called extended $p$-tempered $0$-stable distributions in \cite{Grabchak:2016}. When $p=2$ they are called class $M$, see \cite{Maejima:2015} and the references therein. When $p=1$ we get the Thorin class, which is the smallest class of distributions on $\mathbb R^d$ that is closed under convolution and weak convergence and contains the distributions of all so-called elementary gamma random variables on $\mathbb R^d$, see \cite{Barndorff-Nielsen:Maejima:Sato:2006} or \cite{Laverny:etal:2021}. A related characterization for every $p>0$ is given in Theorem 4.18 of \cite{Grabchak:2016}.
\end{remark}

We now show that the Dickman approximation holds for the small jumps of L\'evy processes associated with GMGD. Toward this end, let $\rho(r,s) = q(r^p,s)r^{-1}$ and $g(s)=1$. Definition \ref{defn: GMGD} implies that $r\rho(r,s) = q(r^p,s)\to1=g(s)$ as $r\downarrow0$. From here, the fact that $r\rho(r,s)=q(r^p,s)\le 1$, implies that the assumptions of Remark \ref{remark:uni2L1} are satisfied and thus that the result of Theorem \ref{thrm: convergence} holds. This means that we can approximately simulate a GMGD L\'evy process by using the approximation in \eqref{eq: approx decomp}. To simulate the process of large jumps, we can either use \eqref{eq: comp pois rep} or \eqref{eq: comp pois rep 2}. Either way, we need a way to simulate from $\tilde\nu^\epsilon_p$. We now develop an approach to do this.

Fix $\epsilon>0$, let $\tilde\nu^\epsilon$ be the finite Borel measure defined by
$$
\tilde\nu^\epsilon(B) = \int_{\mathbb{S}^{d-1}} \int_{\epsilon}^\infty 1_{B}(rs)q(r^p,s) r^{-1} \rd r \sigma(\rd s), \ \ \  B \in \mathfrak{B}(\mathbb{R}^d),
$$
let $\lambda^\epsilon =\tilde\nu^\epsilon(\mathbb R^d)$, and let $\tilde\nu^\epsilon_p =\tilde\nu^\epsilon/\lambda^\epsilon$ be a probability measure. For $u>0$ define
\begin{eqnarray*}
		\ell(u) = \int_u^\infty r^{-1} e^{-r^p } \rd r 
		= \frac{1}{p} \int_{u^p }^\infty y^{-1} e^{-y} \rd y 
		= \frac{1}{p} \Gamma(0, u^p),
    \end{eqnarray*}
where $\Gamma(\cdot , \cdot)$ is the upper incomplete gamma function. Next, for $s\in\mathbb S^{d-1}$, define
	\begin{eqnarray*}
		k_\epsilon(s) &=& \int_\epsilon^\infty  q(r^p, s) r^{-1} \rd r  \\
		&=& \int_\epsilon^\infty \int_0^\infty r^{-1} e^{-r^p v} Q_s(\rd v) \rd r \\
		&=& \frac{1}{p} \int_0^\infty \int_{\epsilon v^{1/p}}^\infty r^{-1} e^{-r^p} \rd r Q_s(\rd v) 
		=  \int_0^\infty \ell(\epsilon v^{1/p}) Q_s(\rd v).
	\end{eqnarray*}
Next, define a probability measure on $\mathbb S^{d-1}$ by
	\begin{eqnarray*}
		\sigma_p(\rd s) = \frac{k_\epsilon(s)}{\lambda^\epsilon} \sigma(\rd s)
	\end{eqnarray*}
and a family of probability measures on $(0,\infty)$ by
	\begin{equation*}
		G_V(\rd v; s) = \frac{\ell(\epsilon v^{1/p})}{k_\epsilon(s)} Q_s (\rd v),
	\end{equation*}
	where $s\in\mathbb S^{d-1}$ is a parameter. Finally, define a family of probability measures on $(0,\infty)$ by $G_R(\rd r; a) = g_R(r;a) \rd r$, where $a>0$ is a parameter and
    \begin{equation}\label{distofr}
    	g_R(r; a) = \frac{1}{\ell(a)} r^{-1} e^{-r^p} 1_{[r \ge a]}
    \end{equation}
is the probability density function (pdf).

\begin{prop}\label{prop:zdecomp}
Let $S \sim \sigma_p$. Given $S$, let $V\sim G_V(\cdot; S)$ and, given $S$ and $V$, let $R\sim G_R(\cdot; V^{1/p} \epsilon)$. If $W=RV^{-1/p} S$, then $W \sim \tilde{\nu}^\epsilon_p$.
\end{prop}

There is no general approach for simulating from $\sigma_p$ and $G_V$ as they depend on the measures $\sigma$ and $Q_s $. However, see \cite{Grabchak:2020}, \cite{Xia:Grabchak:2022}, and the references therein for discussions of simulation from distributions on $\mathbb S^{d-1}$. We now develop a rejection sampling method to simulate from $G_R$.  Toward this end we introduce two distributions that will serve as our proposal distributions. Let
$$
 h_{1}(x;a,p) = p x^{p-1} e^{a^p -x^p} 1_{[x \ge a]}
 $$
be a pdf, where $a,p>0$ are parameters, and let 
$$
 h_{2}(x;a,p,\beta) = \beta\frac{x^{-1}}{\log(1/a)}  1_{[a \leq x < 1]}  + (1-\beta)p x^{p-1} e^{1-x^p} 1_{[x \geq 1]}
 $$
be a pdf, where $a,\beta \in(0,1)$, and $p>0$ are parameters. We can simulate from these distributions as follows.

\begin{prop}\label{prop:ub density}
Fix $a>0$, $p>0$, $\beta\in(0,1)$, and let $U\sim U(0,1)$. We have
$$
\left( a^p  - \log ( U ) \right)^{\frac{1}{p}} \sim h_{1}(\cdot;a,p).
$$ 
Furthermore, if $a\in(0,1)$, then
$$
 a^{(1 - U/\beta)}1_{[U \leq \beta]} +  \left( 1 - \log(1-U)+\log(1-\beta) \right)^{\frac{1}{p}}1_{[U > \beta]} \sim h_{2}(\cdot;a,p,\beta).
$$
\end{prop}
 
 If $a\ge1$, we can check that
 $$
 g_R(x) \le C_1 h_{1}(x;a,p),\ \mbox{ where }\   C_1 = \frac{1}{e^{a^p}p\ell(a)}.
 $$
Similarly, if $a\in(0,1)$, we can check that
 $$
 g_R(x) \le C_2 h_2(x;a,p,\beta),\ \mbox{ where }\  C_2 = \frac{1}{\ell(a)}\max\left\{\frac{1}{ep(1-\beta)},\frac{\log(1/a)}{\beta} \right\}.
 $$
From here we can derive the following rejection sampling algorithms. Let
$$
\phi_1(x) = x^{-1}1_{[x \ge a^p]}
$$
and let
$$
\phi_2(x) = \frac{1}{\max\left\{\frac{1}{ep(1-\beta)},\frac{\log(1/a)}{\beta} \right\}\left( \beta\frac{e^{x^p}}{\log(1/a)}  1_{[a \leq x < 1]} + (1-\beta) ep x^{p} 1_{[x \geq 1]} \right)}.
$$\\

\noindent{\bf Algorithm 1}: Simulation from $G_R$ when $p>0$ and $a\ge1$.\\
\textbf{Step 1.} Simulate $U_1,U_2 \iid U(0,1)$ and set $X=\left( a^p  - \log ( U_1 ) \right)$.\\
\textbf{Step 2.} If $U_2\le\phi_1(X)$ return $X^{1/p}$, otherwise go back to Step 1.\\

In this case, the probability of rejection on a given iteration is $1/C_1=e^{a^p}p\ell(a)$.\\

\noindent{\bf Algorithm 2}: Simulation from $G_R$ when $p>0$ and $a\in(0,1)$; $\beta\in(0,1)$ is a tuning parameter.\\
\textbf{Step 1.} Simulate $U_1,U_2 \iid U(0,1)$ and set
$$
X= a^{1 - U_1/\beta}1_{[U_1 \leq \beta]} +  \left( 1 - \log(1-U_1)+\log(1-\beta) \right)^{\frac{1}{p}}1_{[U_1 > \beta]}.
 $$
\textbf{Step 2.} If $U_2\le\phi_2(X)$ return $X$, otherwise go back to Step 1.\\

In this case, the probability of rejection on a given iteration is $1/C_2$. We will generally take $a=\epsilon V^{1/p}$ with $\epsilon$ small. Thus, we are most interested in the case when $a\to0$. By l'H\^opital's rule, we have
$$
\lim_{a\to0} \frac{1}{C_2} = \beta\lim_{a\to0} \frac{\ell(a)}{\log(1/a)}= \beta\lim_{a\to0} \frac{ \int_a^\infty r^{-1} e^{-r^p } \rd r }{-\log(a)} = \beta\lim_{a\to0} \frac{ - a^{-1} e^{-a^p } }{-a^{-1}} =\beta.
$$
Thus, for small $a$ we can select a large $\beta$ to get a good performance. In general, one can select whichever value of $\beta$ maximizes the acceptance probability in a given situation. For simplicity, throughout this paper we take $\beta=1/2$, which leads to a reasonable performance in the situations considered. We now summarize our algorithm to simulate $\tilde X^\epsilon$, the compound Poisson process of large jumps. It combines \eqref{eq: comp pois rep 2} with Algorithms 1 and 2.\\

\noindent\textbf{Algorithm 3:} For $\epsilon>0$, simulate $\tilde X^\epsilon$, the compound Poisson process of large jumps, up to time $T>0$. Fix the tuning parameter $\beta\in(0,1)$.
     \begin{enumerate}
        \item[I.] Simulate $N \sim \mathrm{Pois} (T\lambda^\epsilon)$.
        \item[II.] Simulate $U_1,U_2,\dots,U_N\iid U(0,1)$.
        \item[III.] For $i = 1, 2, \dots, N$:
            \begin{enumerate}
                \item[1.] Simulate $S_i \sim \sigma_p$.
                \item[2.]  Given $S_i$, simulate $V_i \sim G_V(\cdot;S)$.
            	\item[3.] Given $S_i$ and $V_i$, simulate $R_i \sim G_R(\cdot;V^{1/p}\epsilon)$ as follows:
                    \begin{enumerate}
                        \item[(a)] If $\epsilon V_i^{1/p} \ge 1$: 
                         	\begin{enumerate} 
                        		\item[(a.1)] Generate $U'_1, U'_2 \iid U(0,1)$ and set $X_i = \left( \epsilon^p V_i - \log U'_1  \right)$. \\
                         		\item[(a.2)] If $U'_2 \leq X_i^{-1}$ set $R_i = X_i^{1/p}$, otherwise go back to (a.1).
                        \end{enumerate}
                        	\item[(b)] If $\epsilon V_i^{\frac{1}{p}} < 1$: 
                         \begin{enumerate} 
                        		\item[(b.1)] Generate $U'_1, U'_2 \iid U(0,1)$. 
                        		\item[(b.2)] If $U'_1\le\beta$ set $X_i = \left(\epsilon V_i^{1/p}\right)^{1 - U'_1/\beta}$, otherwise set $$X_i=\left[ 1 - \log(1-U'_1) + \log(1-\beta) \right]^{1/p}.$$
                         \item[(b.3)]  If $U'_2 \leq \phi_2(X_i)$ set $R_i = X_i$, otherwise go back to (b.1). 
                        \end{enumerate}
                    \end{enumerate}
            	\item[4.]  Set $W_i = R_i V_i^{-1/p} S_i$.
            \end{enumerate}
        \item[IV.] For any $t\in[0,T]$, set $\tilde X^\epsilon_t = \sum_{i=1}^{N} W_i 1_{[0, t/T]}(U)$.
     \end{enumerate}

We now give a small simulation study to illustrate the performance of this algorithm and the approximation in \eqref{eq: approx decomp}. For simplicity, we focus of the bivariate case with $p = 1$, drift $\gamma=0$, $Q_{s} = \delta_{1}$ for each $s\in\mathbb S^{d-1}$, and we let $\sigma$ be a discrete uniform probability measure on $n$ evenly spaced points in $\mathbb S^1$. Specifically, we take $\sigma = \frac{1}{n} \sum_{i=1}^n \delta_{s_i}$, where $s_i=(\cos\theta_i,\sin\theta_i)$ with $\theta_i = \frac{2\pi}{n} (i-1)$, $i = 1, 2, \dots, n$. In this case $\lambda^\epsilon = k_\epsilon(s) = \ell(\epsilon) = \Gamma(0, \epsilon)$ and the L\'evy measure in \eqref{eqn:PT0S} simplifies to
\begin{equation*}
	\nu(B) = \frac{1}{n} \sum_{i=1}^n \int_0^{\infty} 1_{B}(rs_i) e^{-r}r^{-1} \rd r , \ \ \  B \in \mathfrak{B}(\mathbb{R}^d).
\end{equation*} 
Here, $G_V(\rd v;s) = \delta_1(\rd v)$, which means that $V\sim G_V(\cdot;s)$ if and only if $V = 1$ with probability $1$. Next, note that $\sigma_p =\sigma$, which means that $\sigma_p$ is discrete uniform and we can simulate from it using a standard approach. For concreteness we take $n=30$ to be the number of points in the support of $\sigma_p$. To simulate from $G_R$ we take the tuning parameter $\beta=1/2$.

Since we are in the bivariate case, we write $X_t = (X_{1,t},X_{2,t})$ and $\tilde X^\epsilon_t = (\tilde X^\epsilon_{1,t},\tilde X^\epsilon_{2,t})$ to denote the GMGD L\'evy process and the compound Poisson process of large jumps, respectively, at time $t$. It is readily checked that
$$
\begin{array}{ll}
    	\rE\left[\tilde X_{1,t}^{\epsilon}\right] = t e^{-\epsilon} \frac{1}{n} \sum_{i=1}^n \cos \theta_i,& \rE\left[\tilde X_{2,t}^{\epsilon}\right] = t e^{-\epsilon} \frac{1}{n} \sum_{i=1}^n \sin \theta_i, \\
    	\mathrm{Var}\left(\tilde X_{1,t}^{\epsilon}\right) = t (\epsilon + 1) e^{-\epsilon} \frac{1}{n} \sum_{i=1}^n \cos^2 \theta_i , &
    	\mathrm{Var}\left(\tilde X_{2,t}^{\epsilon}\right) = t (\epsilon + 1) e^{-\epsilon} \frac{1}{n} \sum_{i=1}^n \sin^2 \theta_i ,
	    \end{array}
	    $$
	    and 
	    $$
	        	\mathrm{Cov}\left(\tilde X_{1,t}^{\epsilon}, \tilde X_{2,t}^{\epsilon}\right) = t (\epsilon + 1) e^{-\epsilon} \frac{1}{n} \sum_{i=1}^n \cos \theta_i \sin \theta_i.
	    $$
We can similarly calculate the means, variances, and the covariance on the components of $X_t$. The formulas are the same, but with $\epsilon=0$. Note that all of these quantities scale linearly in $t$.

   \begin{figure}\label{fig: paths}
	\center
	\begin{tabular}{c}
		\includegraphics[width=0.7\textwidth]{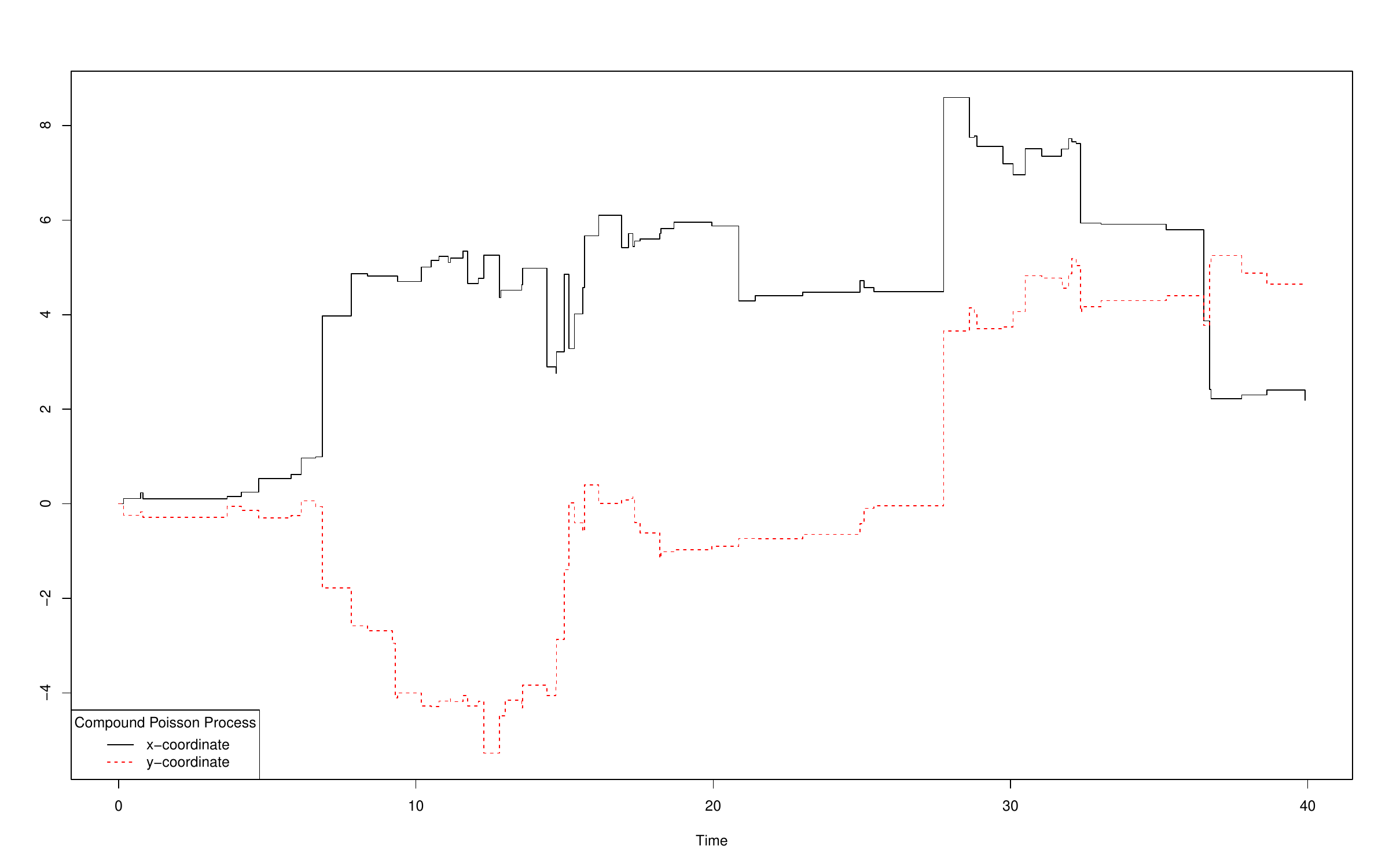}\vspace{-.2cm} \\ (a)\\
		 \includegraphics[width=0.7\textwidth]{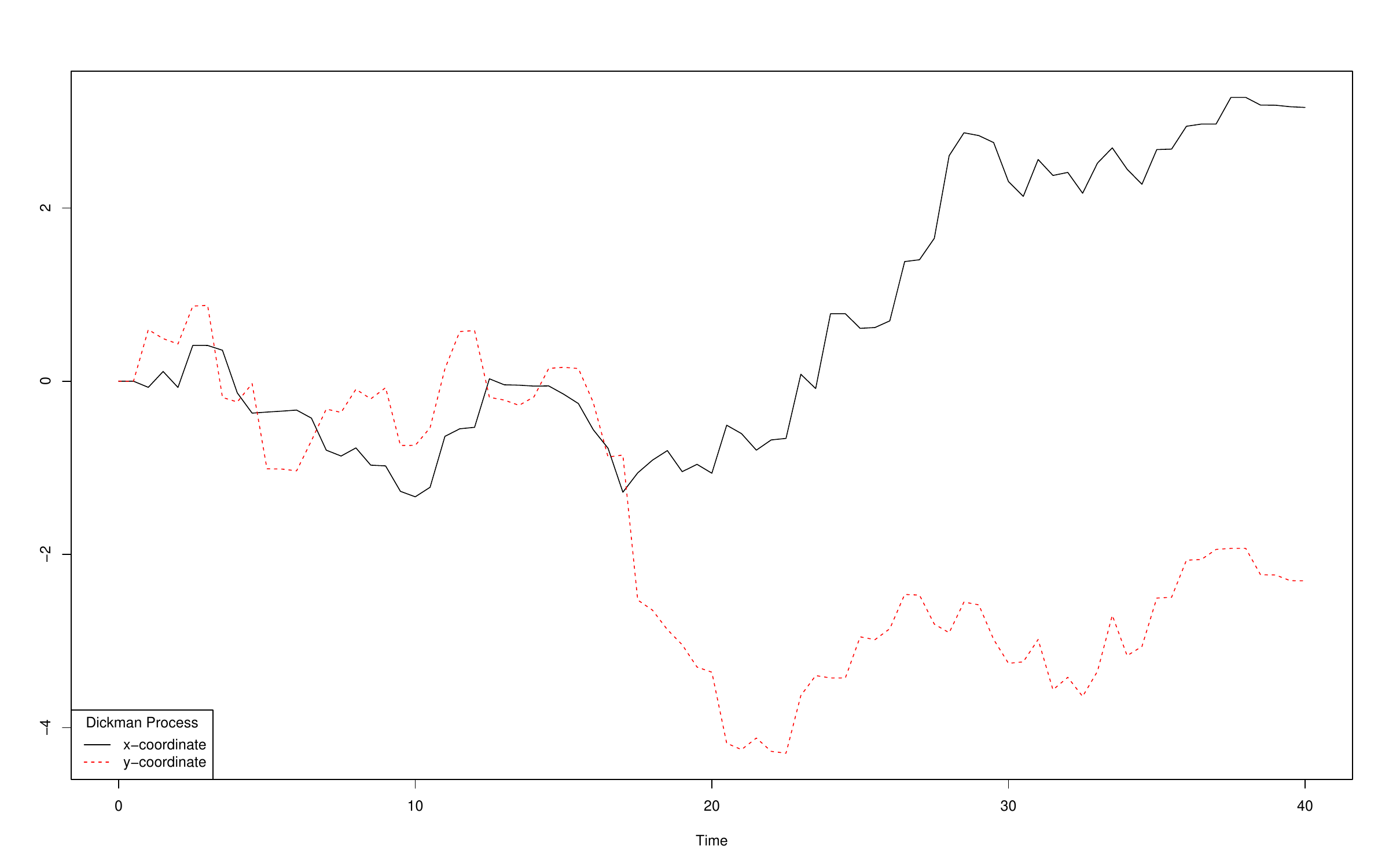} \vspace{-.2cm}\\
				(b)\\
		\includegraphics[width=0.7\textwidth]{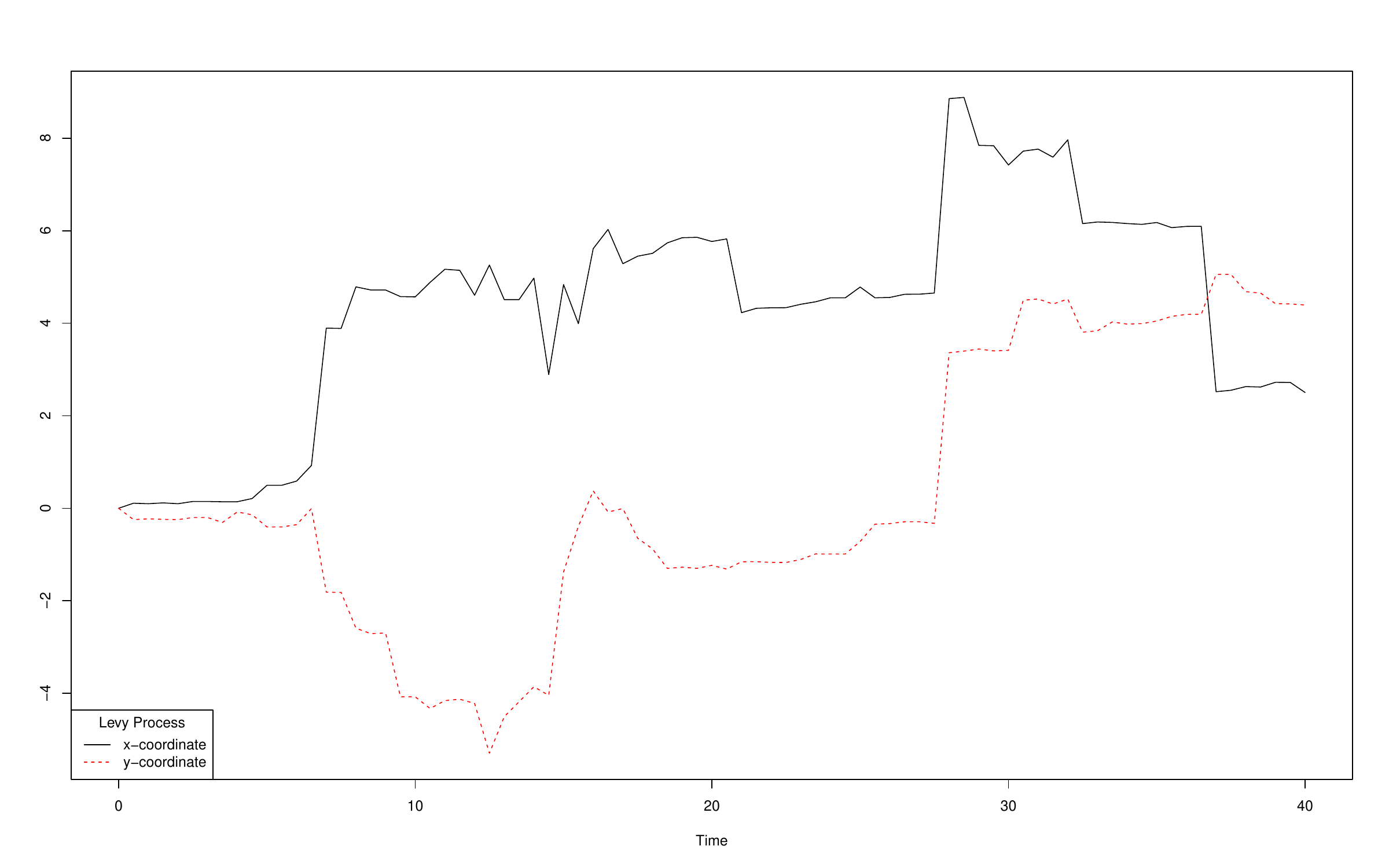} \vspace{-.2cm}\\
		(c)
	\end{tabular}
	\caption{(a) gives a simulated path of $\tilde X^\epsilon$, the compound Poisson process of large jumps. (b) gives a simulated path of $Y^1$, the multivariate Dickman L\'evy process used to approximate the small jumps. (c) gives the simulated path of $X = \tilde X^\epsilon+\epsilon Y^1$. In plots  (a) and (c) we take $\epsilon=0.1$.}
\end{figure}

In Figure 1 we plot a sample path of the process. First, in Figure 1(a) we plot a path of $\tilde X^\epsilon$, the compound Poisson process of large jumps, which was simulated using Algorithm 3. Then, in Figure 1(b) we plot a path of $Y^1$, the Multivariate Dickman L\'evy process used to approximate the small jumps. This was simulated using \eqref{eq: sim MD process}, where we truncate the infinite sum at $10,000$. Finally, in Figure 1(c) we plot a path of $X = \tilde X^\epsilon+\epsilon Y^1$. This path is based on the paths presented in Figures 1(a) and 1(b). In these simulations we take $\epsilon=0.1$.

Next, we performed a small simulation study to better understand the error in our approximation. Toward this end we simulated $N=500,000$ paths of the L\'evy process $X$. For each time $t$, let $m_1(t)$ and $s^2_1(t)$ and $m_2(t)$ and $s^2_2(t)$ be the sample means and sample variances for the first and second components, respectively, and let $s_{1,2}(t)$ be the sample covariance. Now, set
    \begin{eqnarray*}
        \text{ErrMean}_i(t) &=& \frac{\left|\rE\left[X_{i,t}\right] - m_i(t)\right|}{t}, \ \ i=1,2\\
        \text{ErrVar}_i(t) &=& \frac{\left|\mathrm{Var}\left(X_{i,t}\right) - s^2_i(t)\right|}{t}, \ \ i=1,2\\
        \text{ErrCov}(t) &=& \frac{\left|\mathrm{Cov}\left( X_{1,t}, X_{2,t}\right) - s_{1,2}(t)\right|}{t}
    \end{eqnarray*}
 to be the errors in our estimates. Note that we divide by $t$ since the theoretical values scale linearly in $t$. We then combine these into one total error term given by
 $$
 \mathrm{TotalError}(t) = \left( \mathrm{ErrMean}_1(t)^2 + \mathrm{ErrMean}_2(t)^2  + \mathrm{ErrVar}_1(t)^2  + \mathrm{ErrVar}_2(t)^2 + \mathrm{ErrCov}(t)^2 \right)^{1/2}.
$$
Furthermore, to understand the performance of Algorithm 3 in simulating $\tilde X^\epsilon$, we performed a similar simulation study. We again used $N=500,000$ paths and quantified the error analogously, but now using the formulas for the means, variances, and the covariance that are appropriate for this process.

The results of these simulations are presented in Figure 2. In Figure 2(a), we can see that the error in simulating $\tilde X^\epsilon$ is small for all $\epsilon$'s considered. This is not surprising as Algorithm 3 is exact for all choices of $\epsilon$. In Figure 2(b) we see that our approximate method for simulating the process $X$ works well for small $\epsilon$. We note that the difference in the error between $\epsilon=0.1$ and $\epsilon=0.01$ is small, suggesting that $\epsilon=0.1$ is a good choice for this process. In Figure 2(c) we fix $\epsilon=0.1$ and compare the performance of our approach of taking $X \approx \tilde X^\epsilon+\epsilon Y^1$ against a potential approach of just removing the small jumps and taking $X\approx\tilde X^\epsilon$. We can see that the approach where we model the small jumps using a Dickman L\'evy process has significantly less error.

\begin{figure}
	\center
	\begin{tabular}{c}
		\includegraphics[width=0.7\textwidth]{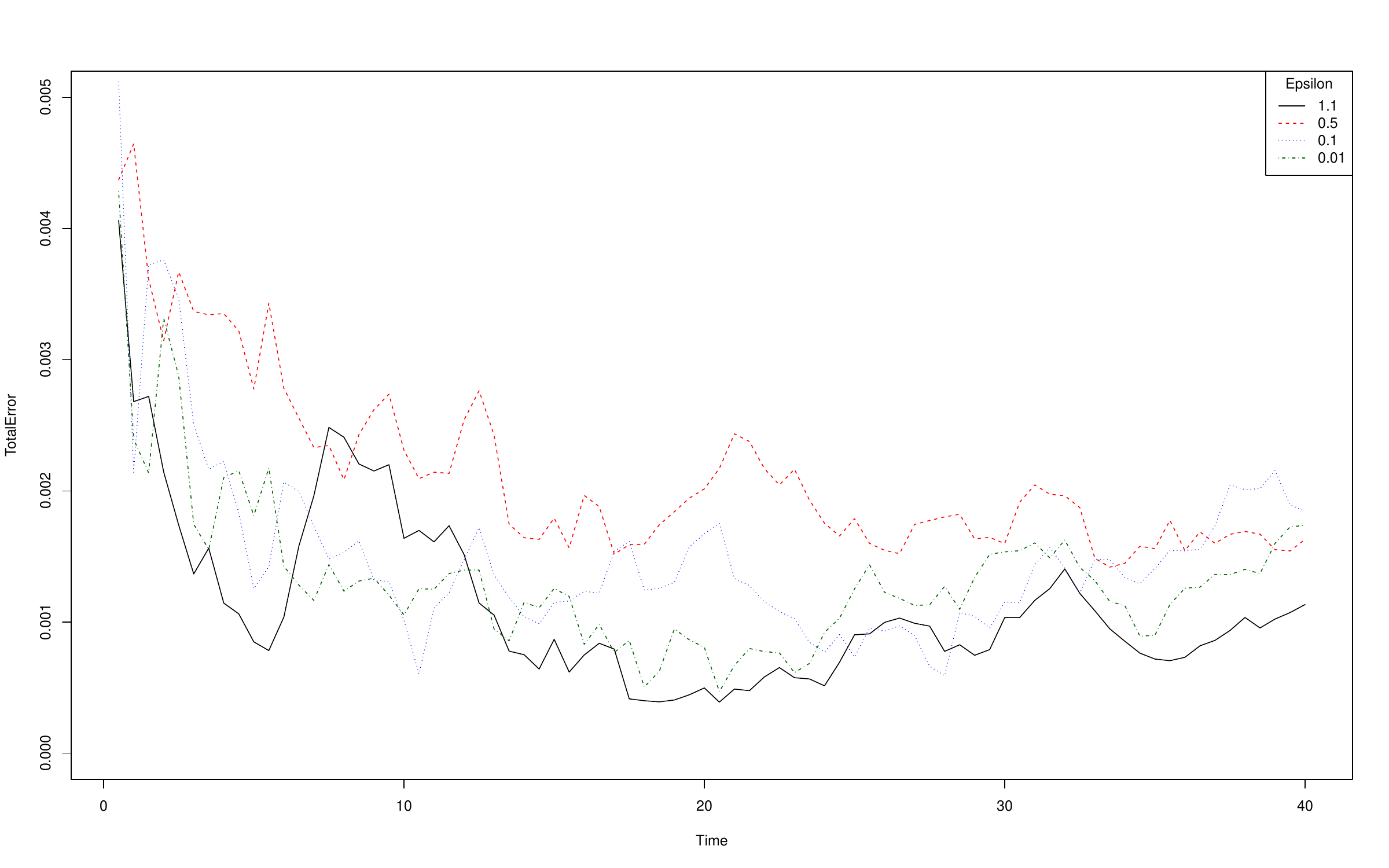} \vspace{-.2cm}\\ (a)\\
		 \includegraphics[width=0.7\textwidth]{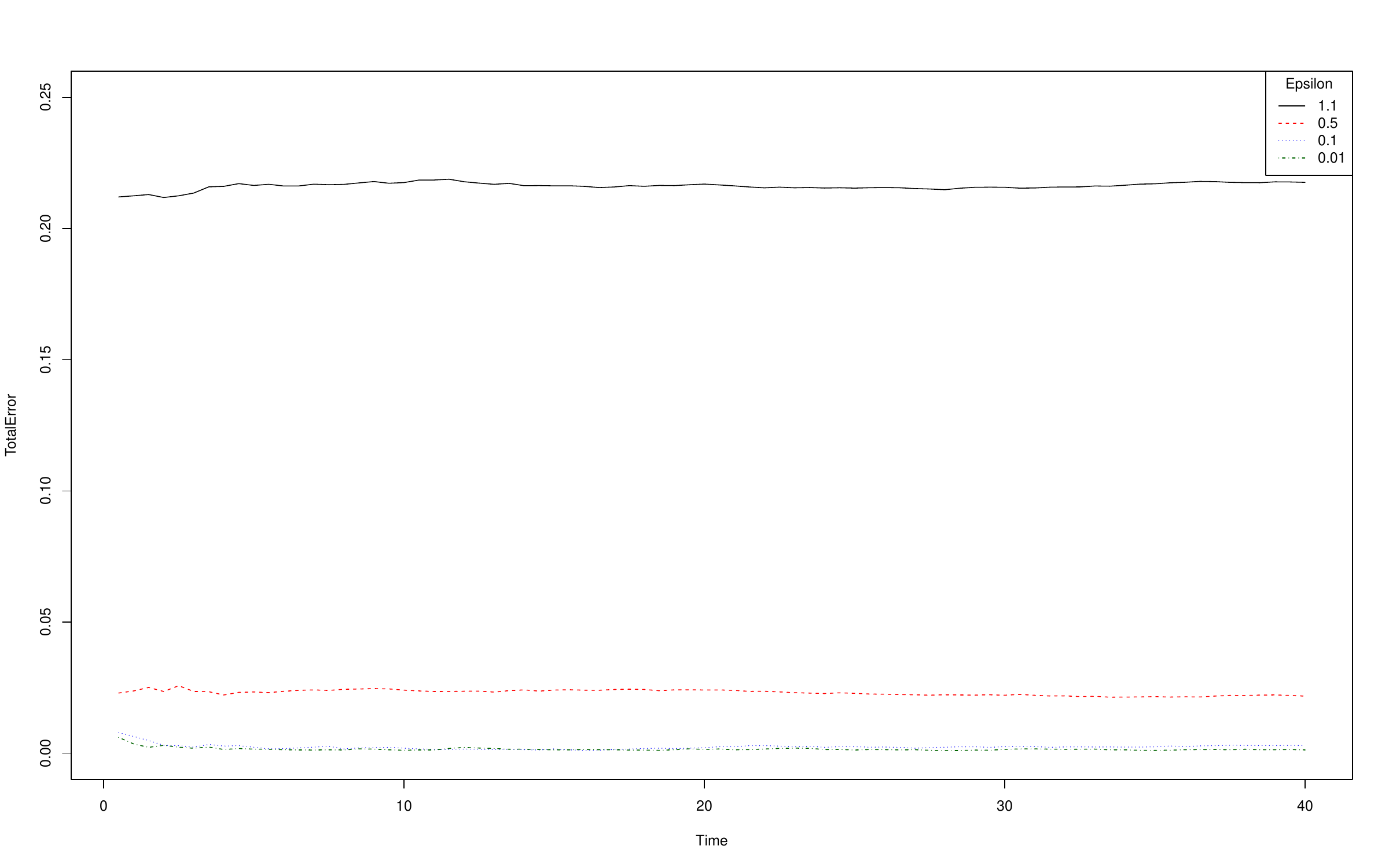} \vspace{-.2cm}\\
		(b)\\
		 \includegraphics[width=0.7\textwidth]{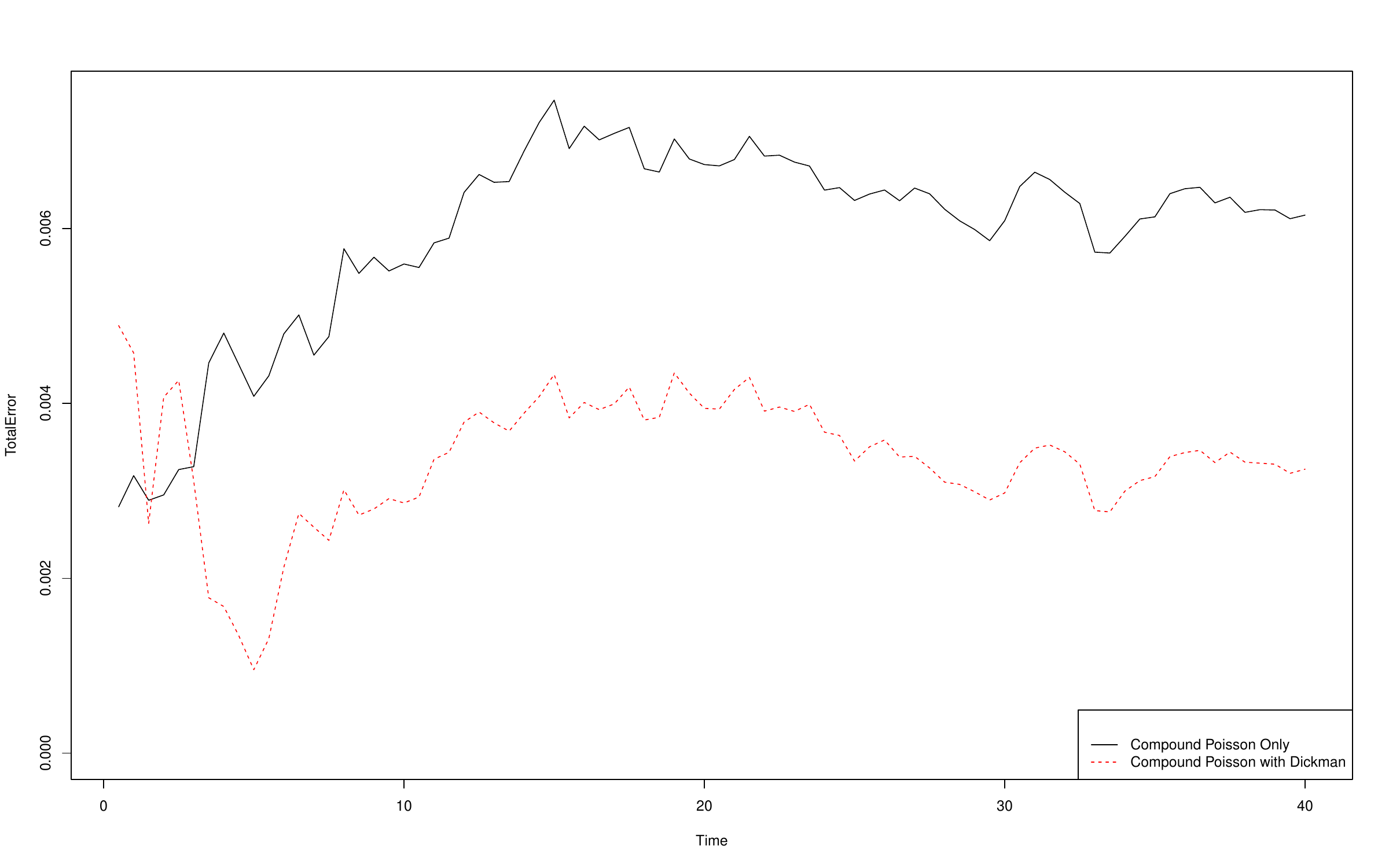}\vspace{-.2cm} \\
		(c)\vspace{-.2cm}
	\end{tabular}
	\caption{Plots of errors. (a) gives the error when simulating $\tilde X^\epsilon$, the process of large jumps, for several choices of $\epsilon$. (b) gives the error when simulating $X$, the L\'evy process of interest, for several choices of $\epsilon$.  (c) compares the error when simulating $X$ by just the process of large jumps with the sum of the process of large jumps and the Dickman approximation to the small jumps. Here we take $\epsilon=0.1$. All plots are based on $N=500,000$ Monte Carlo replications.}
\end{figure}

\section{Proofs}\label{sec: proofs}

\begin{proof}[Proof of Lemma \ref{lemma: Levy measure of MD}]
For the first part, let $Y=X/\gamma$ and note that, by \eqref{eq: relation for MD eps}, 
$$
Y = \frac{1}{\gamma}X \eqd  U^{1/\theta}\left(\frac{1}{\gamma}X+ \frac{\epsilon}{\gamma} \xi\right) = U^{1/\theta}\left(Y+ \frac{\epsilon}{\gamma} \xi\right),
$$
which implies that $X^\epsilon/\gamma=Y\sim \MD^{\epsilon/\gamma}(\sigma)$. For the second part, let $X\sim\mu$ and let $C_\mu(z)$, $z\in\mathbb R$ be the cgf of $\mu$.
By the first part of this lemma, $X/\epsilon\sim \MD(\sigma)$, which is the standard multivariate Dickman distribution and by Theorem 5 in \cite{Grabchak:Zhang:2024} we have $\MD(\sigma)=\ID_0(D^1, 0)$,
where $D^1$ is as in \eqref{eq: Levy meas epsilon} with $\epsilon=1$.  It is easily checked that the cgf of $X/\epsilon$ is given by $C_\mu(z/\epsilon)$, $z\in\mathbb R$.
From here \eqref{eq: char func inf div finite variation}
implies that the cgf of $X/\epsilon$ is given, for any $z \in \mathbb{R}^d$, by
\begin{eqnarray*}
	C_\mu(z/\epsilon)&=& \int_{\mathbb{S}^{d-1}}\int_0^1\left(e^{i\left\langle z, s\right\rangle r} - 1\right) r^{-1}\rd r \sigma(\rd s)
\end{eqnarray*}
and thus that
\begin{eqnarray*}
	C_\mu(z)&=& \int_{\mathbb{S}^{d-1}}\int_0^1\left(e^{i\left\langle z, s\right\rangle r\epsilon} - 1\right) r^{-1}\rd r \sigma(\rd s)\\
	&=& \int_{\mathbb{S}^{d-1}}\int_0^\epsilon\left(e^{i\left\langle z, s\right\rangle r} - 1\right) r^{-1}\rd r \sigma(\rd s) = \int_{\mathbb{R}^d}\left(e^{i\left\langle z, x\right\rangle } - 1\right) D^\epsilon(\rd x),
\end{eqnarray*}
which the the cgf of $ \ID_0(D^\epsilon, 0)$ as required. Note that, in the above, the second equality follows by change of variables.  
\end{proof}

\begin{proof}[\textbf{Proof of Proposition \ref{prop:equivalent condition}}] 
The equivalence between Conditions 1 and 2 follows easily from Lemma 4.9 in \cite{Grabchak:2016}. We just note that for every $h>0$ we have $D^1(\{x\in\mathbb R^d:|x|=h\})=0$, 
$$
\nu((\epsilon h, \epsilon]C) = M^\epsilon\left(\left\{x\in\mathbb R^d:|x|>h, \frac{x}{|x|}\in C\right\}\right),
$$
and
$$
\sigma(C)\log \frac{1}{h} = D^1\left(\left\{x\in\mathbb R^d:|x|>h, \frac{x}{|x|}\in C\right\}\right).
$$

We now show that Condition 1 implies Condition 3. First fix $p,\epsilon>0$, $N \in \mathbb{N}$,  $C \in \mathfrak{B}(\mathbb{S}^{d-1})$ with $\sigma(\partial C)=0$, and let $\epsilon_k = 2^{-(k-1)} \epsilon$. We have
		\begin{eqnarray*}
		\int_{(2^{-N},1]C} |x|^p M^\epsilon(\rd x) &=&
			\frac{1}{\epsilon^p}\int_{(\frac{\epsilon}{2^N},\epsilon]C}|x|^p\nu(\rd x) \\
			&=& \sum \limits_{k=1}^{N}\frac{1}{\epsilon^p} \int_{\big(\frac{\epsilon}{2^k},\frac{\epsilon}{2^{k-1}}\big]C} |x|^p \nu(\rd x)\\
		&=& \sum \limits_{k=1}^{N}\frac{1}{2^{p(k-1)} \epsilon^p_k} \int_{\big(\frac{1}{2}\epsilon_k ,\epsilon_k\big]C} |x|^p \nu(\rd x)\\
			&=&\sum \limits_{k=1}^{N} 2^{-p(k-1)} \int_{(\frac{1}{2},1]C} |x|^p M^{\epsilon_k}(\rd x).
		\end{eqnarray*}
By the version of the Portmanteau Theorem given in \cite{Barczy:Pap:2006}, for any $h\in(0,1)$ we have
$$
\int_{(h,1]C} |x|^p M^\epsilon(\rd x)\to \int_{(h,1]C} |x|^p D^1(\rd x)= \int_{C}\int_h^1 r^{p-1}\rd r\sigma(\rd s) =\frac{\sigma(C)}{p}(1-h^{p})
$$
as $\epsilon\downarrow0$. It follows that, for every $\theta>0$ there exists a $\delta>0$ such that if $0<\epsilon<\delta$, then
$$
\left| \int_{(h,1]C} |x|^p M^\epsilon(\rd x)- \frac{\sigma(C)}{p}(1-h^{p})\right|\le\frac{\theta}{\sum_{k=1}^\infty 2^{-p(k-1)}}.
$$
Taking $h=1/2$, $0<\epsilon<\delta$, noting that $\epsilon_k\in(0,\epsilon]$, and applying the triangle inequality gives
$$
\left| \int_{(2^{-N},1]C} |x|^p M^\epsilon(\rd x)- \sum_{k=1}^N \left(\frac{1}{2^p}\right)^{k-1} \frac{\sigma(C)}{p}(1-1/2^{p})\right|\le \theta.
$$
Now taking the limit as $N\to\infty$ and applying monotone convergence gives
$$
\left| \frac{1}{\epsilon^p}\int_{(0,\epsilon]C} |x|^p \nu (\rd x) - \frac{\sigma(C)}{p}\right|=\left| \int_{(0,1]C} |x|^p M^\epsilon(\rd x) - \frac{\sigma(C)}{p}\right|\le \theta.
$$
Condition 3 is now proved. Note that we did not assume $ \frac{1}{\epsilon^p}\int_{(0,\epsilon]C} |x|^p \nu (\rd x) <\infty$. The fact that this is finite is part of the result.

It is immediate that Condition 3 implies Condition 4. We now show that Condition 4 implies Condition 2, which will complete the proof. Define $\eta_{\epsilon}$ and $\eta$ to be finite Borel measures on $\mathbb R^d$ such that, for $B \in \mathfrak{B}(\mathbb{R}^d)$, we have
	\begin{equation*}
		\eta_{\epsilon}(B) = \int_{B} |x|^p  M^{\epsilon}(\rd x)\ \mbox{ and } \ \eta(B) = \int_{B} |x|^p D^1(\rd x).
	\end{equation*}
Fix $C \in \mathfrak{B}(\mathbb{S}^{d-1})$ with $\sigma(\partial C)=0$,  $h>0$, and $\epsilon>0$. Let $h_0=h\wedge1$ and note that
	\begin{eqnarray*}
		\eta_{\epsilon}\left((0,h]C \right)=\int_{(0,h]C} |x|^p M^{\epsilon}(\rd x) 
		=\int_{(0,\epsilon h]C}\frac{|x|^p}{\epsilon^p} \nu^{\epsilon}(\rd x)
		=\frac{1}{\epsilon^p} \int_{(0,\epsilon h_0]C} |x|^p \nu(\rd x) 
	\end{eqnarray*}
and
	\begin{eqnarray*}
		\eta\left((0,h]C \right)=\int_{(0,h]C} |x|^p D^1(\rd x) 
		= \int_{C} \int_{0}^{h_0} |rs|^pr^{-1} \rd r \sigma (\rd s)
		=h_0^p \frac{\sigma(C)}{p},
	\end{eqnarray*}
where we use the fact that $|s|=1$ for every $s\in\mathbb S^{d-1}$. Since Condition 4 holds
	\begin{eqnarray*}
		\lim\limits_{\epsilon \downarrow 0}\eta_{\epsilon}\left((0,h]C \right)
		=h_0^p\lim\limits_{\epsilon \downarrow 0} \frac{1}{h_0^p\epsilon^p} \int_{(0,\epsilon h_0]C} |x|^p\nu(\rd x) 
		=h_0^p\frac{\sigma(C)}{p}
		=\eta\left((0,h]C\right).
	\end{eqnarray*}
Hence, if $0<h<1$
	\begin{eqnarray*}
		\eta_{\epsilon}\left((h,\infty)C\right)=\eta_{\epsilon}\left((h,1]C\right)&=&\eta_{\epsilon}\left((0,1]C\right)-\eta_{\epsilon}\left((0,h]C\right)\\
		&\rightarrow& \eta\left((0,1]C\right)-\eta\left((0,h]C\right)
	=\eta\left((h,1]C\right) = \eta\left((h,\infty)C\right)
	\end{eqnarray*}
and if $h\ge1$
	\begin{eqnarray*}
		\eta_{\epsilon}\left((h,\infty)C\right)=0 = \eta\left((h,\infty)C\right).
	\end{eqnarray*}
From here Lemma 4.9 in \cite{Grabchak:2016} implies that $\eta_{\epsilon}\conv\eta$ as $\epsilon\downarrow0$. Next, noting that $\frac{1}{|x|^p}$ is bounded and continuous when away from zero and applying the Portmanteau Theorem of \cite{Barczy:Pap:2006} show that, for any $h\in(0,1)$,
	\begin{eqnarray*}
		\lim_{\epsilon\downarrow0} \nu((\epsilon h,\epsilon])  = \lim_{\epsilon\downarrow0}\int_{(h,1]C}\frac{1}{|x|^p}\eta_{\epsilon}(\rd x) = \int_{(h,1]C}\frac{1}{|x|^p}\eta(\rd x)= D^1\left((h,1]C\right) = \sigma(C) \log\frac{1}{h},
	\end{eqnarray*}
which gives the result.
\end{proof}

\begin{proof}[\textbf{Proof of Theorem \ref{thrm: convergence}}]
We only prove the first part, as the second part follows immediately by combining the first part with Slutsky's Theorem. Theorem 15.17 in \cite{Kallenberg:2002} implies that $\epsilon^{-1}{X^{\epsilon}} \cond  Y^1$ if and only if $\epsilon^{-1}{X_1^{\epsilon}} \cond  Y^1_1$. By a version of  Theorem 15.14 in \cite{Kallenberg:2002} (see also Theorem of 8.7 of \cite{Sato:1999} or Theorem 3.1.16 in \cite{Meerschaert:Scheffler:2001}) we have $\epsilon^{-1}X^{\epsilon}_1 \cond  Y^1_1$ if and only if the following three conditions hold:
	\begin{itemize}
\item [\textbf{1.}] $M^{\epsilon}\conv D^1$ as $\epsilon\downarrow0$;
\item[\textbf{2.}] $\lim\limits_{\delta \downarrow 0} \limsup\limits_{\epsilon \downarrow 0} \int_{|x| \leq \delta} \langle z,x \rangle^2 M^{\epsilon}(\rd x) =0$ for every $z\in\mathbb R^d$;
\item[\textbf{3.}] $\lim\limits_{\epsilon \downarrow 0}\int_{|x|\le 1} x M^{\epsilon}(\rd x) = \int_{|x|\le1} x D^1(\rd x)$.
	\end{itemize}
Condition 1 is one of the equivalent conditions in Proposition \ref{prop:equivalent condition}. To complete the proof, we will show that it implies Conditions 2 and 3. To show that it implies Condition 2, note that
 \begin{eqnarray*}
		0 &\leq& \lim\limits_{\delta \downarrow 0} \limsup_{\epsilon \downarrow 0} \int_{|x| \leq \delta} \langle z,x \rangle^2 M^{\epsilon}(\rd x) \\
		&\leq& \lim\limits_{\delta \downarrow 0} \limsup_{\epsilon \downarrow 0} \int_{|x| \leq \delta} |z|^2 |x|^2 M^{\epsilon}(\rd x) \\
		&\leq& \lim\limits_{\delta \downarrow 0} \limsup_{\epsilon \downarrow 0} |z|^2 \delta \int_{|x| \leq 1} |x| M^{\epsilon}(\rd x) \\
		&=&|z|^2 \lim\limits_{\delta \downarrow 0}  \limsup_{\epsilon \downarrow 0}\delta \int_{|x| \leq \epsilon} \frac{|x|}{\epsilon} \nu(\rd x) \\
		&=&|z|^2\sigma(\mathbb{S}^{d-1})\lim\limits_{\delta \downarrow 0}\delta =0,
	\end{eqnarray*}	
where the last line follows by Proposition \ref{prop:equivalent condition}. Condition 3 can be shown in a manner similar to how we showed that Condition 1 implies Condition 3 in the proof of Proposition \ref{prop:equivalent condition}. There are two main differences. First, we take $\int_C s\sigma(\rd s)$ instead of $\sigma(C)$. Second, we no longer take the norm of $x$ in the integral and must use the dominated convergence theorem instead of monotone convergence when taking the limit in $N$. We can use dominated convergence here since we have assumed that $\int_{|x|\le 1} |x| \nu(\rd x)<\infty$.
\end{proof}

\begin{proof}[\textbf{Proof of Corollary \ref{cor:density}}]
Note that for any $C \in \mathfrak B(\mathbb{S}^{d-1})$
	\begin{eqnarray*}
		\lim_{\epsilon \downarrow 0} \frac{1}{\epsilon} \int_{(0,\epsilon]C} |x| \nu(\rd x) &=& \lim_{\epsilon \downarrow 0} \frac{1}{\epsilon} \int_0^{\epsilon} \int_C r \rho(r,s) \sigma(\rd s) \rd r\\
		&=&  \lim_{\epsilon \downarrow 0} \frac{1}{\epsilon} \int_0^{\epsilon} \int_C \left(r \rho(r,s)- g(s)\right)\sigma(\rd s) \rd r+ \lim_{\epsilon \downarrow 0} \frac{1}{\epsilon} \int_0^{\epsilon} \int_C g(s) \sigma(\rd s) \rd r\\
		&=&\int_C g(s) \sigma(\rd s) = \sigma_g(C),
	\end{eqnarray*}	
where the last line follows from the fact that $\left|\int_C (r \rho(r,s) - g(s)) \sigma(\rd s) \right|\leq \int_{\mathbb{S}^{d-1}} |r \rho(r,s) - h(s)| \sigma(\rd s)\to 0$ as $r\downarrow0$. From here the result follows by Proposition \ref{prop:equivalent condition} and Theorem \ref{thrm: convergence}.
\end{proof}

\begin{proof}[\textbf{Proof of Corollary \ref{cor: decomp}}] 
The idea of the proof is similar to that of Theorem 3.1 in \cite{Cohen:Rosinski:2007}. By Theorem \ref{thrm: convergence}, $\frac{X^{\epsilon}}{\epsilon} \overset{d}{\rightarrow} Y^1$, and so, by Theorem 15.17 of \cite{Kallenberg:2002}, there exists a family of L\'evy processes $\hat X^\epsilon = \{\hat X_t^{\epsilon}: t\geq0\}$, $\epsilon>0$, such that 
\begin{equation*}
\hat X^{\epsilon} \overset{d}{=} \frac{X^{\epsilon}}{\epsilon}
\end{equation*}
for each $\epsilon>0$, and satisfying
\begin{equation*}
 \sup\limits_{t\in[0,T]}\left |\hat X_t^{\epsilon}- Y^1_t\right | \conp 0\ \mbox{ as }\ \epsilon \downarrow 0
	\end{equation*}
for each $T>0$. We can (and will) take $\hat X^{\epsilon}$ to be independent of $\tilde X^\epsilon$, possibly on an enlarged probability space, e.g., where $Y^1$ and $\hat X_t^{\epsilon}$ depend on different coordinates from $\tilde X^\epsilon$. Next, for $t>0$ set
	\begin{equation*}
		Z^{\epsilon}_t=\epsilon\left(\hat X_t^{\epsilon}- Y^1_t\right)
	\end{equation*}
and note that \eqref{eq: Zt epsilon to 0} holds for each $T>0$. Note further, that $\{Z_t:t\ge0\}$ is a c\`adl\`ag process, since it is the difference of two L\'evy processes and L\'evy processes have c\`adl\`ag  paths. We have
	\begin{eqnarray*}
		X = X^{\epsilon}+\tilde X^{\epsilon}+\gamma^* 
		\eqd \epsilon \hat X^{\epsilon}+\tilde X^{\epsilon}+\gamma^*
		=\epsilon Y^1+\tilde X^{\epsilon}+\gamma^*+Z^{\epsilon},
	\end{eqnarray*}
which completes the proof.
\end{proof}

\begin{proof}[\textbf{Proof of Proposition \ref{prop:zdecomp}}]
    For any $B \in \mathfrak{B}(\mathbb{R}^d)$, by the tower property of conditional expectation
        \begin{eqnarray*}
        \rP (W \in B) &=& \rE\left[\rE[\rE[ 1_B(RV^{-1/p}S) | V,S] |S]\right]\\
        &=& \rE\left[\rE\left[ \frac{1}{\ell(V^{1/p}\epsilon)}\int_{\epsilon V^{1/p}}^\infty 1_B(rV^{-1/p}S)r^{-1}e^{-r^p} \rd r \Bigm| S\right]\right]\\
        &=& \rE\left[\rE\left[ \frac{1}{\ell(V^{1/p}\epsilon)}\int_\epsilon^\infty 1_B(rS)r^{-1}e^{-r^pV} \rd r \Bigm| S\right]\right]\\
        &=& \rE\left[ \int_0^\infty\frac{1}{\ell(v^{1/p}\epsilon)}\int_\epsilon^\infty 1_B(rS)r^{-1}e^{-r^pv} \rd r\frac{\ell(v^{1/p}\epsilon)}{k_\epsilon(S)} Q_S (\rd v) \right]\\
        &=& \frac{1}{\lambda^\epsilon}  \int_{\mathbb{S}^{d-1}}  \int_0^\infty\int_\epsilon^\infty 1_B(rs)r^{-1}e^{-r^pv} \rd r Q_s (\rd v) \sigma(\rd s) =  \tilde{\nu}^\epsilon_p(B)
        \end{eqnarray*}
        as required.
\end{proof}

\begin{proof}[\textbf{Proof of Proposition \ref{prop:ub density}}]
It is easily checked that the cumulative distribution function (cdf) corresponding to $h_{1}(\cdot;a,p)$ is $H_{1}(x;a,p) = 1 - e^{a^p - x^p}$ for $x>a$. Thus, $H_{1}^{-1}(x;a,p) = \left[ a^p - \log \left(1 - x \right) \right]^{\frac{1}{p}}$ for $x\in[0,1]$. From here, the fact that $U\eqd 1-U$ gives the first part.

Next we turn to simulation from $h_{2}$. Note that this distribution is a mixture. With probability $1-\beta$ we must simulate from $h_{1}(\cdot;1,p)$, which we already know how to do. With probability $\beta$ we must simulate from a distribution with pdf 
$
\frac{x^{-1}}{\log(1/a)}  1_{[a \leq x <1]}.
$ 
The corresponding cdf is given by $1-\log (x)/\log(a)$ for $a<x<1$. The inverse function is then $a^{1-x}$ for $x\in[0,1]$. To conclude, we note that the conditional distribution of $U/\beta$ given $U\le\beta$ is $U(0,1)$ and that the conditional distribution of $(1-U)/(1-\beta)$ given $U>\beta$ is $U(0,1)$.
\end{proof}

\end{document}